\documentclass[11pt,a4paper]{article}
\usepackage{natbib,amsmath,amssymb,amsthm}
\usepackage{enumerate}

\title{Harmonic Analysis of Symmetric Random Graphs \\
\emph{\normalsize Dedicated to the memory of Franti\v{s}ek Mat\'u\v{s}}}
\author{Steffen Lauritzen\\ University of Copenhagen}

\newtheorem{thm}{Theorem}
\newtheorem{lem}{Lemma}
\newtheorem{corr}{Corollary}

\newcommand{\dd}{\mathrm{d}}
\newcommand{\spa}{\mathcal{X}}

\newcommand{\R}{\mathbb{R}}

\newcommand{\N}{\mathbb{N}}
\newcommand{\E}{\mathbb{E}}

\begin{document}
\maketitle
\begin{abstract}This note %Following \cite{ressel:85,ressel:08} this note 
attempts to understand graph limits as defined by \cite{lov06} %,borgs_convergent_2008,lov12} 
in terms of 
harmonic analysis on semigroups. This is done by representing probability distributions of random exchangeable graphs as mixtures of characters on the semigroup of unlabeled graphs with node-disjoint union,  %\citep{berg:etal:84}, 
thereby providing an alternative derivation of de Finetti's theorem for random exchangeable graphs.%; see also \cite{lauritzen:88,diaconis:janson:08}; and \cite{lauritzen:08}.
\\[\baselineskip]

\end{abstract}
\section{Introduction}
Random exchangeable graphs are fundamental for the analysis of network data, see for example \cite{orbantz:roy:15} and \cite{lauritzen:rinaldo:sadeghi:18}. \cite{diaconis:janson:08} showed that the modern theory of graph limits \citep{lov06, lov12} gave a natural way of understanding such graphs. 

Some of the arguments associated with establishing the connection between exchangeable graphs and graph limits can appear complicated. However, as earlier demonstrated by \cite{ressel:85,ressel:08}, the theory of positive definite functions on Abelian semigroups provides a simple, generic structure for understanding exchangeability, and this is what this note attempts to exploit and explain in somewhat larger detail.  Also, this analysis establishes that the graphon based models for random exchangeable graphs can be understood as natural `exponential families' for random graphs, in contrast to the so-called exponential random graph models (ERGMs) \citep{holland:leinhardt:81,snijders:06}.

An ultra short summary of the main points in the developments below has appeared as part of \cite{lauritzen:rinaldo:sadeghi:19}.

\section{Harmonic analysis on semigroups}
We shall use the following concepts from \cite{berg:etal:76}, see also \cite{berg:etal:84}. We consider an Abelian semigroup $(S,+)$ with neutral element $0$ and identity involution.  A function $\rho: S\to \R$ is a \emph{character} if and only if
\[\rho(s+t)=\rho(s)\rho(t),\quad \rho(0)=1\]that is, a character behaves like an exponential function. 
A function $\phi:S\to \R$ is \emph{positive definite} if and only if 
\[\sum_{j,k=1}^n c_jc_k\phi(s_j+s_k)\geq0, \quad \forall n\in \N, c_j\in \R, s_j\in S;\] in other words, a function is positive definite if and only if all matrices of the form $\{m_{jk}=\phi(s_j+s_k)\}$ are positive semidefinite.

We let $\mathcal{P}(S)$ denote the set of positive definite functions on $S$, $\mathcal{P}^b(S)$ the set of all bounded positive definite functions, and $\mathcal{P}^b_1$ the set of bounded positive definite functions $\phi$ with $\phi(0)=1$. When equipped with the topology of pointwise convergence, $\mathcal{P}^b_1$ is a compact Hausdorff space. 

We let $\hat{S}$ denote the set of all bounded characters on $(S,+)$ and note that these form an Abelian semigroup under multiplication; $(\hat S,\cdot)$ is the \emph{dual semigroup} to $(S,+)$.
We shall be using the following result from \cite{berg:etal:76}.
\begin{thm}[Berg, Christensen, Ressel]\label{thm:bcr76}The set $\mathcal{P}^b_1$ is a Bauer simplex with the bounded characters as extreme points. In particular, any $\phi\in\mathcal{P}^b_1$  has a unique representation as  barycentre of a probability measure $\mu$ on $\hat S$:
\[\phi(s)=\int_{\hat S} \rho(s)\, \dd\mu(\rho)\]
and the set of bounded characters $\hat S$ form a closed subset of $\mathcal{P}^b_1$.
\end{thm}
\section{Generalized exponential families}
Generalized exponential families were introduced in \cite{lauritzen:75} and studied further in \cite{lauritzen:88} where Theorem~\ref{thm:bcr76} was exploited to identify these as so-called \emph{extreme point models} for i.i.d.\ observations. 

More precisely, a generalized exponential family of distributions on a discrete state space $\spa$ is determined as a set of probability mass functions of the form
\[p(x;\theta) = \frac{\theta(t(x))}{c(\theta)}b(x),\quad  \theta \in \Theta\subseteq \tilde S, \quad x\in \spa\]
where $t: \spa\to S$ is a \emph{canonical sufficient statistic} with values in an Abelian semigroup $S$, $b$ defines a \emph{base measure}, $\tilde S$ are the (not necessarily bounded if $b$ is not uniform) characters, and $\Theta$ is the canonical parameter space
\[\Theta=\left\{\theta\in \tilde S: c(\theta) = \sum_{x\in \spa}\theta(t(x))b(x)<\infty\right\}.\]
Note that $c(\theta)$ is the \emph{Laplace transform} of the lifted base measure $b^*$ on $S$:
\[b^*(s)= \sum_{x\in t^{-1}(s)}b(x), \quad c(\theta)= \sum_{s\in S} \theta(s)b^*(s).\]
These exponential families share many (but not all) properties of more standard exponential families which have the same form but for the special semigroup $(\R^k,+)$, usually considered as a vector space.
\section{Graphs}
For any $n \in \N$ we let $\mathcal{L}_n$ and
$\mathcal{U}_n$ denote
the set of simple labeled graphs and simple unlabeled graphs with node set
$[n] := \{1,\ldots,n\}$; further, we let $\mathcal{L}_0=\mathcal{U}_0$ be the empty graph, $\mathcal{L}_\infty$ the set of infinite graphs with node set $\N$,  $\mathcal{L} = \bigcup_{n=0}^\infty
\mathcal{L}_n$, and $\mathcal{U}=\bigcup_{n=0}^\infty
\mathcal{U}_n$.

For  $G\in \mathcal{L}$  % and $G'$ in $\mathcal{L}$, 
we %$G \sim G'$ to signify
   % that they are isomorphic and 
   let $[G]$ denote the equivalence class of all labeled graphs
    isomorphic to $G$.  We will also without ambiguity think of $[G]$ as an unlabeled graph.

If $G$ and $H$ are in $\mathcal{L}$, we write $H \subseteq G$ if $H$ is a subgraph of $G$.  If $A$ is a subset of the node set in a labeled graph $G$, the subgraph induced by $A$ has $A$ as node set and edge set equal to all edges in $G$ wiht both endpoints in $A$. For $m \leq n$ and $G \in \mathcal{L}_n$,
$G[m]$ is the subgraph of $G$ induced by $[m]$. An infinite graph in $\mathcal{L}_\infty$ can be identified with its sequence of induced subgraphs $G_n=G([n]), n\in \N$; such a sequence is \emph{consistent} in the sense that for all $m\leq n$ we have $G_m=G_n([m])$.

    For $n<\infty$, $G \in \mathcal{L}_n$, and $\sigma \in \mathcal{S}_n$, the permutation
    group on $[n]$, we will let $G_\sigma$ be the graph obtained from $G$ by
    relabeling its nodes according to $\sigma$. Thus $i \sim j$ in $G$ if and
    only if $\sigma(i) \sim \sigma(j)$ in $G_\sigma$. 
\section{Symmetric random graphs}
We consider a probability distribution $P$ on $\mathcal{L}_\infty$ and say this is \emph{symmetric} if and only if 
\[P(G[n]=H)=P(G[n]= H_\sigma), \quad \forall n\in \N, H\in \mathcal{L}_n, \sigma\in S_n.\]
We then also say the random graph $G$ or its distribution $P$ is \emph{exchangeable} \citep{aldous:81,aldous:85,diaconis:freedman:81,matus:95,diaconis:janson:08,lauritzen:08}.

It is practical to represent the distribution $P$ through its \emph{M\"obius parameters} \citep{drt08,lauritzen:rinaldo:sadeghi:18,lauritzen:rinaldo:sadeghi:19} $Z(F), F\in \mathcal{L}^*$ where $\mathcal{L}^*=\bigcup_2^\infty \mathcal{L}^*_n $ is the set of labeled graphs with no isolated nodes and
\[Z(F) = P(G: F\subseteq G).\]
The quantities $Z$ and $P$ are related by the M\"obius transform; if $G\in \mathcal{L}_n$
\begin{equation}\label{eq:mobius}  P(G)  =  \sum_{B\in \mathcal{L}^*:G\subseteq B\subseteq K_n}  (-1)^{|B\setminus G|}  Z(B), \quad Z(F)=\sum_{B\in\mathcal{L}_n:F\subseteq B}P(B)\end{equation}
where $K_n$ is the complete graph on $n$ nodes and $|B\setminus C|$ denotes the cardinality of the differencs of the corresponding edge set.  A non-negative function $Z$ is a valid M\"obius parameter if and only if the left-most expression in (\ref{eq:mobius}) is non-negative for all $G\in \mathcal{L}$. We note that the positivity condition is not so easy to verify for a given function $Z$, so it is of interest to derive alternative representations.

Clearly, $P$ is symmetric if and only if $Z$ is symmetric, or, equivalently, if there is a function $\phi:\mathcal{U}\to \R$ such that
\begin{equation}\label{eq:symmetry}Z(F)= \phi([F]).\end{equation}
We also note that $(\mathcal{U},+)$ --- where $U+V=U\cup V$ is (node disjoint) graph union --- is an Abelian semigroup with the empty graph as its neutral element.  So the map 
$t: \mathcal{L}\to \mathcal{U}$  which maps any finite graph into its equivalence class: $t(F)=[F]$, appears as the canonical sufficient statistic in the generalized exponential family of exchangeable random graphs.
Moreover we have
\begin{lem}\label{lem:positivity}Let $G$ be a random exchangeable graph with M\"obius parameter $Z$ given as above. Then the function $\phi$ is bounded and positive definite on $(\mathcal{U},+)$; in other words, $\phi\in\mathcal{P}^b_1(\mathcal{U})$.
\end{lem}
\begin{proof}Clearly $\phi(\emptyset)=1$ and $\phi$ is bounded. Introduce the binary random variables $X_{ij}$ for $i\neq j\in \N$ where
$X_{ij}=1$  if $i\sim j$ in $G$ and $X_{ij}=0$ otherwise; $X$ is the (random) adjacency matrix of $G$.  Then, clearly
\[Z(F)= \E\left(\prod_{ij: i\sim j \in F}X_{ij}\right).\]
So elementary calculations will verify that for $F_u$ and $F_v$ being node-disjoint we have
\begin{eqnarray*}\sum_{u,v=1}^n c_uc_v\phi([F_u]+[F_v])&=&
\sum_{u,v=1}^n c_uc_v\E\left(\prod_{i\sim j\in F_u}X_{ij}\prod_{i\sim j\in F_v'}X_{ij}\right)\\&=&
\E\left\{\sum_{u}c_u\prod_{i\sim j\in F_u}X_{ij}\right\}^2\geq 0,
\end{eqnarray*}
where $F'_v$ is a representative of $[F_v]$ which is node-disjoint from $F_u$ for all $u,v$.
This completes the proof.
\end{proof}
We note that the property in Lemma~\ref{lem:positivity} is referred to as \emph{reflection positivity} in \cite{lov06}.
A de Finetti type representation of random symmetric graphs can now be obtained  as a Corollary to Theorem~\ref{thm:bcr76}:
\begin{corr}[deFinetti's theorem for exchangeable random graphs]\label{cor:dissoc}Let $P$ be the distribution of an random graph with M\"obius parameter $Z$. Then $P$ is exchangeable if and only if there is a unique probability measure $\mu$ on $\hat{\mathcal{U}}$ such that for all $F\in \mathcal L$
\[Z(F)=\int_{\hat {\mathcal{U}} }
\rho([F])\, \dd \mu(\rho).\]
\end{corr}
We note in particular that the extreme points of the convex set of exchangeable measures --- corresponding to the pure characters --- are \emph{dissociated} \citep{silverman:76}, i.e if $F=F_1\cup F_2$ and $F_1$ and $F_2$ are node disjoint subgraphs of $F$ it holds that
\[Z(F)= \rho([F])=\rho([F_1])\rho([F_2])=Z(F_1)Z(F_2)\]
or, in other words, we have for node-disjoint $F_1,F_2$:
\[P_\rho(F_1\cup F_2\subseteq G)=P_\rho(F_1\subseteq G)P_\rho(F_2\subseteq G).\]
\section{Characters on the semigroup of unlabeled graphs}
To get a more detailed version of de Finetti's theorem for exchangeable graphs we need to identify the characters on $(\mathcal{U},+)$ that enter into the representation above since not all bounded characters will be valid in the sense that their M\"obius transform would be non-negative. We shall say that such characters are \emph{fully positive} and denote the set of these characters by $\hat{\mathcal{U}}_+$.
We consider first the \emph{homomorphism densities}
\[ t_{\mathrm{hom}}(F,G) = \frac{\mathrm{hom}(F,G)}{ |G|^{|F|}},\]
where $F,G\in \mathcal{L}$ and $\mathrm{hom}(F,G)$ is the number of graph homomorphisms (edge preserving maps) from $F$ to $G$ and, as before $|F|$ and $|G|$ is the cardinality of the edge sets of $F$ and $G$.
These are multiplicative in the sense that for node disjoint subgraphs $F_1, F_2$
\begin{equation}\label{eq:hom_char}t_{\mathrm{hom}}(F_1\cup F_2,G) = t_{\mathrm{hom}}(F_1,G)t_{\mathrm{hom}}( F_2,G).\end{equation}
Noting that in fact $t_{\mathrm{hom}}(F,G)$ only depends on $(F,G)$ through their isomorphism classes $([F],[G])$ we can 
for $[G]\in \mathcal{U}$ define a character as
\begin{equation}\label{eq:basecharacter}\rho_{[G]}([F])=t_{\mathrm{hom}}(F,G).\end{equation}
In addition we shall need the \emph{injective homomorphism densities}
\[ t_{\mathrm{inj}}(F,G) = \frac{\mathrm{inj}(F,G)}{ (|G|)_{|F|}},\]
where $F,G\in \mathcal{L}$ and $\mathrm{inj}(F,G)$ is the number of injective graph homomorphisms (edge preserving maps) from $F$ to $G$,  and $$(x)_y=x(x-1)\cdots (x-y+1)$$ is the falling factorial.

The homomorphism densities can be understood as M\"obius parameters of a probability distributions of a random graph $F$, where vertices in $F$  are sampled  from $G$ with or without replacement respectively. If we let $p_{\mathrm{hom}}(F,G)$ and $p_{\mathrm{inj}}(F,G)$ denote the corresponding probability distributions, we have therefore  the inequalities 
\begin{equation}\label{eq:sampling}
\sup_{F\in \mathcal{L}_m}|p_{\mathrm{hom}}(F,G)- p_{\mathrm{inj}}(F,G) | \leq 1- \frac{(|G|)_{|F|}}{|G|^{|F|}}\leq \frac{{{|F|}\choose{2}}}{|G|},
\end{equation}
and therefore in particular
\begin{equation}\label{eq:samplingmob}
|t_{\mathrm{hom}}(F,G)- t_{\mathrm{inj}}(F,G) | \leq 1- \frac{(|G|)_{|F|}}{|G|^{|F|}}\leq \frac{{{|F|}\choose{2}}}{|G|},
\end{equation}
see \cite{freedman:77} and Lemma 3.3 in \cite{lauritzen:rinaldo:sadeghi:19} or Lemma 2.1 in \cite{lov06} who give the slightly weaker bound to the right. 

 In particular we note that for any fixed $F$, the right-hand side in (\ref{eq:sampling}) tends to zero as $|G|\to \infty$.

Note now that for  any exchangeable random graph with M\"obius parameter $Z$, the law of total probability yields
\begin{equation}\label{eq:total}Z(F)=\sum_{G\in \mathcal{L}_n}t_{\mathrm{inj}}(F, G)P^n_Z(G),\end{equation}
where $P^n_Z$ the induced measure on $\mathcal{L}_n$, 
We then have
\begin{thm}\label{thm:characters}For every $[G]\in \mathcal{U}$, the function $\rho_{[G]}$ is a bounded and fully positive character in $\hat{\mathcal{U}}_+$. Further, these characters are dense in $\hat{\mathcal{U}}_+$: if we let $\mathcal{B}$ denote the closure of $\{\rho_{[G]},[G]\in \mathcal{U}\}$
within $\mathcal{P}^b_1$ (pointwise convergence), we have $\mathcal{B}=\hat{\mathcal{U}}_+$.
\end{thm}
\begin{proof} For any $[G]$, $\rho_{[G]}$ is clearly a bounded and fully positive character by (\ref{eq:hom_char}). Since the set of bounded characters $\hat{\mathcal{U}}$ is closed by Theorem~\ref{thm:bcr76} and therefore also the set of fully positive characters, we clearly have $\mathcal{B}\subseteq \hat{\mathcal{U}}_+$. We need to establish the reverse inequality.  

  Thus let  $Z\in\hat{ \mathcal{U}}_+$ and let $P_Z$ be the corresponding probability measure on $\mathcal{L}_\infty$ with $P^n_Z$ the induced measure on $\mathcal{L}_n$. By (\ref{eq:sampling}) and (\ref{eq:total}) it then holds for any $F\in \mathcal{L}_m$ and any $n\geq m$ that
\begin{equation}\label{eq:approx}
Z(F)=\sum_{G\in \mathcal{L}_n}t_{\mathrm{hom}}(F, G)P^n_Z(G) + R(m,n),\end{equation}
where
\[|R(m,n)|\leq  \sum_{G\in \mathcal{L}_n}|t_{\mathrm{hom}}(F,G)- t_{\mathrm{inj}}(F,G) | P^n_Z(G)\leq 1- \frac{(n)_{m}}{n^{m}}.\]
Using (\ref{eq:basecharacter}) we can rewrite (\ref{eq:approx}) as 
\begin{equation}\label{eq:finite_deF}Z(F) = \int_{\mathcal{B}}\rho_{[G]}(F)\,d\mu^n_Z([G])+R(m,n)\end{equation}
which now holds for all $n \geq m$. 

Since $\mathcal{B}\subseteq \mathcal{U}$ is compact, so is the set of probability measures on $\mathcal{B}$ and hence the sequence $\mu^n_Z, n\in \N$ must have an accumulation point $\mu^*_Z$ which then satisfies for all $n\in \N$:
\[Z(F) = \int_{\mathcal{B}}\rho_{[G]}(F)\,d\mu^*_Z([G]).\]
Now, as the integral representation is unique by Theorem~\ref{thm:bcr76} and Corollary~\ref{cor:dissoc}, we must have $\mu^*_Z$ concentrated on $Z\in \hat{\mathcal{U}}_+$ and thus $\hat{\mathcal{U}}_+\subseteq \mathcal{B}$ which completes the proof.
\end{proof}
Note that, in effect, (\ref{eq:finite_deF}) yields a finite deFinetti type representation for an exchangeable random graph, see also \cite{diaconis:freedman:80} and \cite{lauritzen:rinaldo:sadeghi:19}. We state this results as its own corollary and note that this is in fact Theorem $1^{j.e}$ of \cite{matus:95}:
\begin{corr}[deFinetti's theorem for finitely exchangeable random graphs]Let $\mu_m$ be the induced distribution of  $G[m]$ where $G$ is a finitely exchangeable random graph in $\mathcal{L}_n$. Then there is an exchangeable random graph $G^*$ such that the distribution $\mu^*_m$ of $G^*[m]$ satisfies
\[||\mu_m-\mu_m^*||_\infty\leq 2R(m,n)\leq \frac{m(m-1)}{n}\]
where $||\cdot||_\infty$ is the total variation norm.
\end{corr}
\begin{proof}We define $\mu^*$ by its M\"obius parameter $Z^*$ as
\[Z^*(F) = \sum_{G\in \mathcal{L}_n}t_{\mathrm{hom}}(F, G)P^n_Z(G)\]
or, equivalently
\[\mu^*_m(F) = \sum_{G\in \mathcal{L}_n}p_{\mathrm{hom}}(F, G)P^n_Z(G)\]
and now (\ref{eq:sampling}) yields
\begin{eqnarray*}||\mu_m-\mu_m^*||_\infty &=&2\sup_{F}|\mu_m(F)-\mu^*_m(F)|\\&\leq&
2\sum_{G\in \mathcal{L}_n}|p_{\mathrm{inj}}(F, G)-p_{\mathrm{hom}}(F, G)|P^n_Z(G)\\&
\leq &2 R(m,n)\leq \frac{m(m-1)}{n}
\end{eqnarray*}
which was required.
\end{proof}

Elements of $\mathcal{B}$ can naturally be interpreted as \emph{limits of unlabeled graphs} \citep{lov06,borgs_convergent_2008,lov12} by the embedding $[G]\mapsto \rho_{[G]}$ and 
\[ \lim_{n\to \infty} [G_n]= U \iff \lim_{n\to \infty}\rho_{[G_n]}(F) =U(F), \quad \forall F\in \mathcal{U}\]\enlargethispage{\baselineskip}
so we can write $\mathcal{U}_\infty=\mathcal{B}=\hat{\mathcal{U}}_+$.

In addition,  the characters can be represented by (equivalence classes of) functions $W:[0,1]^2\to[0,1]$   also known as a \emph{graphons}, see references above, corresponding to adjacency matrices of infinitely exhangeable random arrays \citep{aldous:81,hoover:79,diaconis:freedman:81}.
This is contained in the following result:
\begin{thm}The fully positive and bounded characters $\hat{\mathcal{U}}_+$ on $(\mathcal{U},+)$ are exactly the functions $\rho$
satisfying for $F\in \mathcal{L}^*_n$
\[\rho([F])=\int_{[0,1]^n}\prod_{ij: i\sim j \in F}W(u_i,u_j)\,\dd u\]
for some measurable, symmetric function $W:[0,1]^2\to[0,1].$ The function $W$ is unique up to measure-preserving transformations of the unit interval.
\end{thm}
\begin{proof}It is straightforward to see that these are fully positive and bounded characters. Also for every graph $[G]\in \mathcal{U}_n$ we can construct a specific $W_{[G]}$ by partitioning the unit interval into $n$ subintervals \[[0,1)=[0,1/n)\cup \cdots \cup,[(n-1)/n)=A_1\cup \cdots \cup A_n\] and letting for $G$ being a representative of $[G]$
\[W_{[G]}(u,v) = \begin{cases} 1 & \text{ if $i\sim j$ in $G$, $u\in A_i, v\in A_j$}\\ 0 & \text{ otherwise} \end{cases}\] and then show that these converge in a suitable metric exactly when $[G_n]\to U$. We refrain from giving the details of this and refer to \cite{lov06} or \cite{lov12}.
\end{proof}
The graphon representation of a graph limit has its advantages, but also its disadvantages in that it can be difficult to identify exactly when two graphons $W$ and $W'$ are representing the same character $\rho$ since there are many measure-preserving transformations of the unit interval. In that sense, the representation as a character $Z$ is more direct and there is a one-to-one correspondence between $Z$ and the corresponding random graph.  However, in general, it is not so easy to decide whether a given function $Z$ specifies a valid probability distribution, i.e.\ satisfies the positivity restriction in (\ref{eq:mobius}). 
\section*{Acknowledgement}I am grateful to an anonymous referee for sharp and constructive comments, correcting substantial shortcomings in the exposition. 
%\bibliographystyle{apalike}
%\bibliography{strings,jasbib,mat}

\end{document}